\documentclass{amsart}
\usepackage{macros}
\standardsettings
\colorcommentstrue

\draftfalse

\newcommand{\sym}{\mathrm{sym}}

\newcommand{\Rd}{\R^d}
\newcommand{\Rddim}{d}

\renewcommand{\path}{path}
\newcommand{\paths}{paths}

\newcommand{\alphabet}{A}
\newcommand{\dBG}{G}
\newcommand{\set}{U}
\newcommand{\collection}{\CC}
\newcommand{\base}{b}

\begin{document}
\title[de Bruijn sequences and Diophantine approximation on fractals]{Uniformly de Bruijn sequences and symbolic Diophantine approximation on fractals}
\authorlior\authorkeith\authordavid

\begin{Abstract}

Intrinsic Diophantine approximation on fractals, such as the Cantor ternary set, was undoubtedly motivated by questions asked by K. Mahler (1984). One of the main goals of this paper is to develop and utilize the theory of infinite de Bruijn sequences in order to answer closely related questions. In particular, we prove that the set of infinite de Bruijn sequences in $k\geq 2$ letters, thought of as a set of real numbers via a decimal expansion, has positive Hausdorff dimension. For a given $k$, these sequences bear a strong connection to Diophantine approximation on certain fractals. In particular, the optimality of an intrinsic Dirichlet function on these fractals with respect to the height function defined by symbolic representations of rationals follows from these results.

\end{Abstract}


\maketitle

\section{Introduction}

In this paper, we give a novel application of combinatorics to the field of Diophantine approximation. Since we do not assume that the reader is familiar with this field, let us first recall some important concepts and ideas. We refer the reader to Section \ref{DiophantineApproximation} where we rigorously define and discuss these notions.


Classically, the field of Diophantine approximation sought to quantify how well real numbers can be approximated by rationals, weighing the distance to the rational point against some function of its denominator. The inaugural result in the field is Dirichlet's theorem, Theorem \ref{Dirichlet}, which states that every irrational real number has infinitely many rational points $p/q$ that lie within distance $1/q^2$ of it. This result raises the question of whether that function, $1/q^2$, can be improved. That it cannot be, in a sense made precise in Section \ref{DiophantineApproximation}, is due to a result of Liouville, who showed that quadratic irrational numbers, like $\sqrt{2}$, admit no better rate of approximation. In modern terminology, we call such points \emph{badly approximable}. 

A more complete description of the set of badly approximable numbers, in this and related contexts, was the subject of much activity in the early-to-mid twentieth century. Via a characterization of badly approximable numbers in terms of continued fraction expansions one can show that the set of badly approximable numbers is uncountable, but it is also relatively easy to show that this set is a Lebesgue null set \cite[Theorem 1.9 and Corollary 1.6]{Bugeaud}, so we must turn to other notions of ``size''. One such notion, particularly well-suited to disntinguishing between sets of measure zero, is that of \emph{Hausdorff dimension}. Jarn\'ik showed that despite being a Lebesgue null set, the set of badly approximable real numbers has full Hausdorff dimension, so it is still ``large'' in some sense. 

As discussed further in Section \ref{DiophantineApproximation}, the core questions of Diophantine approximation can be formulated in many diverse contexts, essentially whenever we have a complete metric space $X$, a countable dense subset $\mathcal{Q}$, and some notion of ``height'' defined on $\mathcal{Q}$ (this would be the size of the denominator in the classical case above). Over the last decade, a plethora of results regarding Diophantine approximation on fractals have emerged \cite{BFKRW,BFR,EFS,Fishman,FishmanSimmons1,FishmanSimmons2,KLW,KleinbockWeiss1,Weiss}. Many of these results were motivated by the following question(s) posed by K. Mahler in 1984 \cite[\62]{Mahler}: ``How close can irrational elements of Cantor's set be approximated by rational numbers

\begin{enumerate}
\item	in Cantor's set, and
\item	by rational numbers not in Cantor's set?''
\end{enumerate}

In this paper we will restrict our attention to Mahler's first question; see Section \ref{Diophantine} for details. We remark that while in \cite{FishmanSimmons2}, the first- and third-named authors were able to exhibit an optimal Dirichlet function (see Definition \ref{definitionoptimal}) corresponding to Mahler's second question, it seems that finding an analogous answer to his first question is significantly harder, see e.g. \cite{BFR,BugeaudDurand,FishmanSimmons2} for detailed discussions and conjectures regarding this question.

In \cite{FishmanSimmons2}, a new height function was defined on the rational points of the Cantor set (see Section \ref{Diophantine}), and a Dirichlet-type theorem was proven \cite[Corollary 2.2 and its proof]{FishmanSimmons2}. The purpose of this paper is to demonstrate the optimality of that Dirichlet theorem, and give an estimate on the Hausdorff dimension of the set of ``badly approximable'' points. This set, as noted in \cite{FishmanSimmons2}, admits a precise combinatorial description, although at the time we had been unable to exhibit any members belonging to it. In the present paper, we focus on a combinatorially defined subset of the set of badly approximable points, the set of \emph{uniformly de Bruijn sequences}. The existence of uniformly de Bruijn sequences demonstrates the optimality of the Dirichlet function (Theorem \ref{BAintrinsic}), and by estimating from below the Hausdorff dimension of the set of uniformly de Bruijn sequences (Theorem \ref{MainTheoremDeBruijn}), we are able to get a positive lower bound for the Hausdorff dimension of the set of badly approximable points (Corollary \ref{corollaryoptimality}), a first step towards a Jarn\'ik-type result. See Section \ref{Diophantine} for a more nuanced discussion of these points. 

\subsection{Acknowledgements} The first-named author was supported in part by the Simons Foundation grant \#245708. The third-named author was supported in part by the EPSRC Programme Grant EP/J018260/1. The authors would like to thank Joseph Kung for valuable comments on an earlier version of the paper, and Jonah Ostroff for introducing us to the notion of de Bruijn sequences. The authors thank the anonymous referee for valuable comments.

\section{Finite and infinite de Bruijn sequences}
\label{DeBruijn}

Let $\alphabet$ be a finite alphabet of cardinality $k \geq 2$. We recall that a (non-cyclic) \emph{de Bruijn sequence of order $n$} in $\alphabet$ is a sequence $\omega$ of length $k^n + n - 1$ in the alphabet $\alphabet$ that has the property that every sequence of length $n$ in $\alphabet$ appears as a consecutive substring of $\omega$ exactly once. For example, in the alphabet $\{0,1\}$, the sequence $00110$ is a de Bruijn sequence of order 2 while in the alphabet $\{0,1,2\}$, the sequence $00010020110120210221112122200$ is a de Bruijn sequence of order 3.
We say that an infinite sequence $\omega\in \alphabet^\N$ is \emph{infinitely de Bruijn} if the set
\begin{equation}
\label{deBruijn}
B_\omega \df \{n \in \N: \text{the initial segment of $\omega$ of length $k^n + n - 1$ is a de Bruijn sequence of order $n$}\}
\end{equation}
is infinite. We say that $\omega$ is \emph{totally de Bruijn} if $B_\omega = \N$, and \emph{uniformly de Bruijn} if $B_\omega$ has bounded gap sizes. The construction of infinitely de Bruijn sequences goes back to Becher and Heiber \cite{BecherHeiber},\Footnote{Note that in \cite{BecherHeiber}, the phrase ``infinite de Bruijn sequence'' has a different meaning; we do not use that meaning in this paper because it makes an ad hoc distinction between the $k = 2$ case and the $k \geq 3$ case.} who showed that when $k \geq 3$, totally de Bruijn sequences could be constructed recursively by extending each de Bruijn sequence of order $n$ to a de Bruijn sequence of order $(n + 1)$. We shall discuss their method in more detail below. When $k = 2$, it is known that no totally de Bruijn sequence exists, but Becher and Heiber do construct a uniformly de Bruijn sequence such that $B_\omega = 2\N$.

In order to state our main theorem for this section, let us briefly recall the definition and basic properties of the Hausdorff dimension of a fractal\Footnote{The word ``fractal'' normally has a connotative but not a denotative meaning in mathematics; a set is called a fractal if it is ``sufficiently complicated at fine scales''. The Cantor ternary set, i.e. the set of all numbers in $[0,1]$ that can be written in base 3 with only the digits 0 and 2, is a canonical example of a fractal; further examples are given in Subsection \ref{Hutchinson}.} $F\subset \Rd$, see e.g. \cite[Chapters 2-3]{Falconer_book}. Let $\dist$ denote the standard metric on $\Rd$, and let $\diam(\set)$ denote the diameter of a set $\set\subset\Rd$. Fix $\delta >0$ and let $F\subset\Rd$. We say that a countable collection $\{\set_j:j\in\N\}$ of subsets of $\Rd$ is a {\textsl{$\delta$-cover}} of $F$ if $F\subset \bigcup_{j = 1}^{\infty}\set_j$ and $\diam(\set_j)\leq\delta$ for every $j$. For each $s\geq 0$, let
\[
\HH^s_{\delta}(F)\df\inf\Big\{\sum_{j=1}^{\infty} \diam(\set_j)^s\text{ : }\{\set_j:j\in\N\} \text{ is a $\delta$-cover of } F\Big\}.
\]
Then the {\textsl{$s$-dimensional Hausdorff measure}} of $F$ is the number
\[
\HH^s(F)\df\lim _{\delta\to 0}\HH^s_{\delta}(F),
\]
and the {\textsl{Hausdorff dimension}} of $F$ is the number
\[
\HD(F)\df\inf\{s \geq 0:\HH^s(F)=0\}=\sup\{s \geq 0:\HH^s(F)=\infty\}.
\]
It is well known that for every $F\subset\Rd$ we have $0\leq \HD(F)\leq \Rddim$, and that if $\HD(F) > 0$, then $F$ is uncountable, but not vice versa.\Footnote{The set of Liouville numbers on the real line is a standard example of a comeager (and thus uncountable) set of Hausdorff dimension 0.}

We also recall that if $\base \geq 2$ is an integer, then the \emph{base $\base$ expansion} of a number $x\in [0,1]$ is the series
\[
\sum_{i = 1}^\infty \frac{\omega_i}{b^i},
\]
where $\omega_1,\omega_2,\ldots\in \{0,1,\ldots,\base - 1\}$ are chosen so that the value of the series is equal to $x$. This choice is unique unless $x$ is a rational number whose denominator is a power of $\base$, in which case there are exactly two ways in which the infinite word $\omega = \omega_1\omega_2\cdots$ can be chosen.

\ignore{
If $F\subset\Rd$ is bounded and nonempty, denote by $N_{\delta}(F)$ the smallest number of sets of diameter at most $\delta$ needed to cover $F$. We define the \textsl{lower box counting dimension} of $F$ as
\[
\underline{\BD}(F)= \liminf_{\delta\to 0}\frac{\log N_{\delta}(F)}{-\log \delta}\cdot
\]
In practice, it is sometimes easier to calculate the lower box counting dimension than the Hausdorff dimension of a set, and we shall use the following  helpful fact: for any nonempty bounded $F\subset\Rd$,

\begin{equation}
\label{hausdorffleqbox}
\HD(F) \leq \underline{\BD}(F).
\end{equation}
}

\begin{theorem}
\label{MainTheoremDeBruijn}
Fix an integer $\base \geq 2$, and let $C(\base)=\{0,1,\cdots ,\base-1\}$. Fix $\alphabet\subset C(\base)$ such that $k \df \#(\alphabet) \geq 2$. Denote by $\delta$ the Hausdorff dimension of the set $F$ consisting of all numbers that can be written in the form $\sum_{i=1}^{\infty}\frac{\omega_i}{\base^{i}}$ with $\omega_i\in \alphabet$ for every $i\in\N$, i.e. the set of all numbers in $F$ that have at least one base $\base$ expansion composed entirely of digits from
$\alphabet$.\Footnote{\label{footnote1}It is well known that $\delta = \log(k)/\log(\base)$, see Subsection \ref{Hutchinson}.}
Then the set $S$ consisting of all elements of $F$ that have at least one base $\base$ expansion that is uniformly de Bruijn satisfies
\[
0 < \alpha_k \delta \leq \HD(S) \leq  \frac{\log (k!)}{k\log (k)} \delta < \delta,
\]
where
\[
\alpha_k = \begin{cases}
1/49 & k = 2\\
(8\cdot (9\log_4(3) - 1))^{-1} & k = 3\\
\frac{\log(k - 2)!}{k\log(k)} & k \geq 4
\end{cases}.
\]
In particular, $S$ has positive Hausdorff dimension but not full Hausdorff dimension.
\end{theorem}

\noindent Note that for large values of $k$, Stirling's formula gives $\alpha_k \sim \frac{\log(k!)}{k\log(k)} \sim 1 - \frac{1}{\log(k)}$ (where $x\sim y$ means $(1 - x)/(1 - y) \to 1$), and in particular $\alpha_k \to 1$ as $k \to \infty$. Thus $S$ gets closer and closer to having full dimension as the number of allowed digits increases.

\section{Preliminaries}

We begin by recalling some key definitions used in Becher and Heiber's paper, as well as the proof of the well-known BEST\Footnote{An acronym after the people who discovered it: de Bruijn, van Aardenne-Ehrenfest, Smith, and Tutte.} theorem.

\begin{definition}[\cite{BecherHeiber}]
\label{de Bruijn graph}
Given an alphabet $\alphabet$ and an integer $n\in\N$, the \emph{de Bruijn graph} of order $n$ on $\alphabet$ is the directed graph $\dBG = \dBG_n(\alphabet)$ with vertex set $V(\dBG) \df \alphabet^n$ and edge set $E(\dBG) \df \{(\omega,\tau) : \omega_{i + 1} = \tau_i \all i \leq n-1 \}$. Note that every vertex has in-degree and out-degree both equal to $k \df \#(\alphabet)$, for a total of $k^n$ vertices and $k^{n + 1}$ edges.

If $\omega$ is a sequence of length $\ell \geq n$ in $\alphabet$, then the \emph{path induced by $\omega$ on $\dBG$} is the path\Footnote{In this paper a ``path'' in a directed graph is a sequence of vertices such that each pair of consecutive vertices is connected by an edge from the first vertex to the second vertex. The length of a path is the number of such edges, or equivalently the number of vertices minus one (counting multiplicity in both cases). A path is \emph{simple} if all its vertices are distinct except possible the first and last, and \emph{Eulerian} if it contains each edge exactly once.} $\gamma = \gamma_1\cdots\gamma_{\ell - n + 1}$ in $\dBG$ defined by the formula
\[
\gamma_i \df \omega_i \cdots \omega_{i + n - 1} \in V(\dBG).
\]
\end{definition}

\begin{observation}
\label{obs Eulerian}
Let $\omega$ be a sequence of length $\ell_\omega = k^m + m - 1$, and let $\gamma$ be the path induced by $\omega$ on $\dBG_n(A)$. Note that the length of $\gamma$ is $\ell_\gamma = \ell_\omega - n = k^m + m - n - 1$; in particular, $\ell_\gamma = k^{n + 1}$ if $m = n + 1$, and $\ell_\gamma < k^n$ if $m \leq n$. Moreover,
\begin{itemize}
\item[(I)] If $m = n + 1$, then $\omega$ is de Bruijn if and only if $\gamma$ is Eulerian.
\item[(II)] If $m \leq n$ and $\omega$ is de Bruijn, then $\gamma$ is a simple path.
\end{itemize}
\end{observation}

\begin{remark}
If $m = n$ and $\omega$ is de Bruijn, then $\gamma$ is a simple path that visits each vertex exactly once. However, since $\gamma$ starts and ends at different vertices, it is not a Hamiltonian cycle, contrary to \cite[p.931, first para.]{BecherHeiber}. In particular, the edge set of $\gamma$ does not form a regular graph on $V(\Omega)$, as is claimed in \cite[Proof of Lemma 3, last para.]{BecherHeiber}. Consequently, the proof given there is technically incorrect; it can be trivially fixed by adding a step where $\gamma$ is extended to a Hamiltonian cycle; cf. the first two paragraphs of the proof of Corollary \ref{lowerbound4} below. Similar remarks apply to \cite[Proof of Lemma 5, last para.]{BecherHeiber}.
\end{remark}

Now let $X = (V(X),E(X))$ be a directed graph such that for each vertex $x\in V(X)$, the in-degree and out-degree of $x$ are nonzero and equal to each other (though they may depend on $x$). Fix a vertex $x_0 \in V(X)$, and let $\EE$ be the set of Eulerian \paths\ of $X$ that start and end at $x_0$. Note that, unlike standard convention, we consider two Eulerian \paths\ to be different if they are formally different as sequences of vertices even if they are cyclically equivalent. Let $\TT$ be the set of directed spanning trees of $X$ rooted at $x_0$ with edges pointing towards $x_0$.

Since both the conclusion of the BEST theorem and its proof will be important for our argument, we recall them now. We once again remind the reader that our statement differs slightly from the usual one because of our convention about counting Eulerian \paths: we do \emph{not} consider cyclically equivalent paths to be the same. But the difference is easy to quantify: the number of Eulerian paths in each cyclic equivalence class that start and end at $x_0$ is equal to the degree of $x_0$ (we recall that by assumption the in-degree and out-degree are equal). So our count will be off from the conventional one by a factor of $\deg(x_0)$.

\begin{theorem}[BEST theorem]
We have
\begin{equation}
\label{BEST}
\#(\EE) = \#(\TT)\cdot\deg(x_0)\cdot\prod_{x\in V(X)} [\deg(x) - 1]! \;.
\end{equation}
\end{theorem}
\begin{proof}
Let $T \in \TT$ be a directed spanning tree rooted at $x_0$. For each $x\in V(X)$, let $E_x$ denote the set of edges in $X$ with initial vertex $x$, and let $T_x = E(T)\cap E_x$, where $E(T)$ denotes the edge set of $T$. If $x\neq x_0$, then $T_x$ is a singleton, say $T_x = \{v_x\}$, while $T_{x_0} = \emptyset$. Now let $\operatorname{Ord}(S)$ denote the set of total orderings of a set $S$, and note that the cardinality of the set
\[
O(T) \df \prod_{x\in V(X)} \operatorname{Ord}(E_x \butnot T_x)
\]
is exactly $\deg(x_0)\cdot\prod_{x\in V(X)} [\deg(x) - 1]!\;$. Now for each $\oo = (o_x)_{x\in V(X)} \in O(T)$ we let $f(T,\oo)$ be the Eulerian \path\ that starts and ends at $x_0$ defined recursively as follows: Suppose that the points $x_0 = \gamma_0,\gamma_1,\ldots,\gamma_i$ have been defined, and let $x = \gamma_i$. Then the next vertex $\gamma_{i + 1}$ must be chosen so that $\gamma_i \gamma_{i + 1} \in E_x$, but $\gamma_i \gamma_{i + 1} \neq \gamma_j \gamma_{j + 1}$ for all $j < i$. We make this choice so as to minimize $\gamma_i \gamma_{i + 1}$ according to the ordering $o_x$ subject to these restrictions. If the edges of $E_x \butnot T_x$ have been exhausted, then if $x \neq x_0$ we choose the vertex $v_x$, and if $x = x_0$, then we terminate the \path. There is some work to do to show that $f(T,\oo)$ is indeed an Eulerian \path, and that every Eulerian \path\ that starts and ends at $x_0$ can be represented uniquely as $f(T,\oo)$ for some $T\in \TT$ and $\oo \in O(T)$, see e.g. \cite[pp.445-446]{Aigner}. This implies that $f$ is a bijection between $\coprod_{T\in\TT} O(T) = \{(T,\oo) : T\in \TT,\; \oo\in O(T)\}$ and $\EE$, which completes the proof.
\end{proof}

We will also need the following sufficient condition for the right-hand side of \eqref{BEST} to be nonzero:

\begin{lemma}
\label{lemmaexistsT}
If $X$ is connected, then there is at least one directed spanning tree rooted at $x_0$, i.e. $\TT \neq \emptyset$.
\end{lemma}
\begin{proof}
Let $T$ be a maximal directed tree rooted at $x_0$. By the maximality of $T$, there is no edge from any vertex not in $T$ to any vertex in $T$. Since each vertex of $X$ has equal in-degree and out-degree, the number of edges from $V(T)$ to $V(X)\butnot V(T)$ is equal to the number of edges from $V(X)\butnot V(T)$ to $V(T)$, which is equal to zero. Since $X$ is connected, this means that either $V(T) = \emptyset$ or $V(X) \butnot V(T) = \emptyset$. But $x_0 \in V(T)$ by construction, so $V(X) \butnot V(T) = \emptyset$ and thus $T$ is a spanning tree, i.e. $T\in \TT$.
\end{proof}

\section{Proof of Theorem \ref{MainTheoremDeBruijn}}

\subsection{The upper bound}

We begin by establishing the upper bound of Theorem \ref{MainTheoremDeBruijn}. To do this we will use the Hausdorff--Cantelli lemma, a very useful tool for establishing upper bounds on the Hausdorff dimensions of certain sets, see e.g. \cite[Lemma 3.10]{BernikDodson}.  Let $\{\set_j:j\in\N\}$ be a countable collection of sets in $\Rd$, and let $\set$ be the set consisting of those elements of $\Rd$ that belong to infinitely many of the sets $\set_j$ ($j\in\N$). In other words,
\[
S \df \limsup_{j\to\infty} \set_j = \bigcap_{N=1}^\infty \bigcup^\infty_{j=N} \set_j.
\]

\begin{lemma}[Hausdorff--Cantelli Lemma]
\label{lemmaHausdorffCantelli}
Let $\{\set_j:j\in\N\} \subset \Rd$ be a countable collection of sets, and let $S = \limsup_j \set_j$. Fix $s > 0$. If
\begin{equation}
\label{HausdorffCantelli}
\sum^\infty_{j=1} \diam(\set_j)^s < \infty,
\end{equation}
then $\HH^s(S) = 0$ and thus $\HD(\set) \leq s$.
\end{lemma}

It turns out to be convenient to consider a collection $\{\set_j:j\in\N\}$ that naturally splits up into subcollections, say $\{\set_j:j\in\N\} = \bigcup_m \collection_m$ for some sequence of collections $(\collection_m)_{m = 1}^\infty$. In this case, the summability condition \eqref{HausdorffCantelli} is equivalent to the condition
\[
\sum^\infty_{m=1} \cost^s(\collection_m) < \infty,
\]
where
\[
\cost^s(\collection_m) \df \sum_{\set \in \collection_m} \diam(\set)^s
\]
is the \emph{$s$-dimensional cost} of $\collection_m$. Note that $\cost^s(\collection_m)$ should be distinguished from the expression $(\cost^1(\collection_m))^s$, which denotes instead the 1-dimensional cost of $\collection_m$ raised to the power of $s$. The set $S$ can be written in terms of the collections $(\collection_m)_{m = 1}^\infty$ as follows:
\[
S = \limsup_{m\to\infty} \bigcup_{\set \in \collection_m} \set = \bigcap_{N = 1}^\infty \bigcup_{m = N}^\infty \bigcup_{\set\in \collection_m} \set.
\]
In what follows we will abuse terminology somewhat by calling $\cost^s(\collection_m)$ the ``cost'' of the set $S_m \df \bigcup_{\set \in \collection_m} \set$, although strictly speaking, it depends not only on $S_m$ but also on how it is decomposed.

\ignore{
Note that we \emph{could} relabel everything so that we indexed each of the smaller sets individually, in which case we would be considering $\diam^s$, and that the lim sup sets are exactly the same. But doing this would obscure the idea behind the proof and the construction of the sequence, so we opt for a ``cost''-based approach instead.

}

\begin{proof}[Proof of upper bound]
For each $m$, let $S_m$ be the set consisting of all elements of $F$ corresponding to base $\base$ expansions whose initial segments of length $k^m +  m- 1$ are de Bruijn sequences of order $m$ in $\alphabet$. Then the lim sup of the sequence $(S_m)_{m = 1}^\infty$ consists of those elements of $F$ with infinitely de Bruijn base $\base$ expansions. In particular, the set $S$ consisting of those elements of $F$ with uniformly de Bruijn base $\base$ expansions satisfies:
\[
S \subset \limsup_{m\to\infty} S_m = \bigcap_{N=1}^\infty \bigcup^\infty_{m=N} S_m.
\]
By the Hausdorff--Cantelli lemma, if we can find an $s$ such that
\begin{equation}
\label{HCbound}
\sum^\infty_{m=1} \operatorname{cost}^s(S_m) < \infty,
\end{equation}
then we can conclude that $\HD(S) \leq s$. We will show that \eqref{HCbound} holds for all $s > \delta \frac{\log(k!)}{k\log(k)}$.

For each $m$, we view $S_m$ as the union of the collection
\[
\collection_m \df \{S^\omega_m: \omega \text{ is a de Bruijn sequence of order $m$ in the alphabet $\alphabet$} \},
\]
where for each $\omega$, $S^\omega_m$ is the set of points $x \in F$ corresponding to base $\base$ expansions whose initial segments of length $k^m + m - 1$ are equal to $\omega$.
Let $\dBG$ be the de Bruijn graph of order $(m-1)$ on $\alphabet$ (see Definition \ref{de Bruijn graph}), so that $\#(V(\dBG)) = k^{m-1}$. By Observation \ref{obs Eulerian}(I), the collection $\collection_m$ is in bijection with the set of Eulerian \paths\ on $\dBG$. Fix a vertex $x_0 \in V(\dBG)$. We can estimate the number of Eulerian \paths\ starting and ending at $x_0$ via the BEST theorem. Specifically, we have $\prod_{x \in V(\dBG)}(\deg(x) - 1)! = (k-1)!^{\#(V(\dBG))}$, since every vertex $x \in V(\dBG)$ has degree equal to $k$. The number of spanning trees rooted at $x_0$ is at most $k^{\#(V(\dBG)) - 1}$, since an edge must be chosen emanating from each vertex $x \ne x_0$, and each vertex has out-degree $k$. And for the same reason, $\deg(x_0) = k$.
Therefore, the number of Eulerian \paths\ starting and ending at $x_0$ is at most
\[
\#(\TT) \cdot \deg(x_0) \cdot \prod_{x \in V(\dBG)}[\deg(x) - 1]! \leq k^{\#(V(\dBG))-1} \cdot k \cdot (k-1)!^{\#(V(\dBG))} = k!^{\#(V(\dBG))} = k!^{k^{m-1}}.
\]
Since there are $\#(V(\dBG)) = k^{m - 1}$ possible choices for $x_0$, the number of de Bruijn sequences of order $m$ in $\alphabet$ is at most $k^{m - 1}\cdot k!^{k^{m-1}}$.\Footnote{In fact, the exact count for such sequences is known, but we prefer this estimate because it is simpler and yields the same upper bound on the Hausdorff dimension.} Now, if $\omega$ is a de Bruijn sequence of order $m$ in $\alphabet$, then the length of $\omega$ is $k^m + m -1$, and thus the diameter of $S^\omega_m$ is at most $\base^{-k^m -m +1}$. So the $s$-dimensional cost of $S_m$ according to the decomposition $\collection_m$ is at most
\[
k^{m - 1}\cdot k!^{k^{m-1}} \cdot (\base^{-k^m -m +1})^s.
\]
Now fix $\varepsilon> 0$ and set
\begin{equation}
\label{equation s}
s\df \frac1{k} \log_\base(k!) + \varepsilon.
\end{equation}
Then
\begin{align*}
\sum^\infty_{m=1} \operatorname{cost}^s(S_m) &\leq \sum^\infty_{m=1} k^{m - 1} (k!)^{k^{m - 1}} (\base^{-k^m-m+1})^s.
\end{align*}
By the ratio test, this series converges as long as $\lim_{m\to\infty} |a_{m+1}/a_m| < 1$, where $a_m$ denotes the $m$th term. A straightforward computation yields:
\begin{eqnarray*}
|a_{m+1}/ a_m| &= k\cdot \base^{-\varepsilon(k^{m+1} - k^m)} \cdot \base^{-s},
\end{eqnarray*}
which tends to $0$ as $m \to \infty$.

Thus by Lemma \ref{lemmaHausdorffCantelli}, we have
\[
\HD(S) \leq \frac1{k} \log_\base(k!) = \frac{\log(k!)}{k \log(\base)} = \frac{\log(k!)}{k \log(k)}\  \delta,
\]
since $\delta = \log(k)/\log(\base)$ (see Subsection \ref{Hutchinson}).

\ignore{
Therefore the lower box dimension of $S$ is at most
\[
\frac{\log(\#(V(X)) \cdot k!^{\#(V(X))})}{- \log(\base^{-k\#(V(X))})} = \frac{\log(\#(V(X))) + \#(V(X))\log(k!)}{k\#(V(X)) \log(\base)} \stackrel{m \to \infty}{\rightarrow} \frac{\log(k!)}{k \log(\base)} = \delta \frac{\log(k!)}{k\log{k}},
\]
since $\delta = \log(k)/\log(\base)$, (see Footnote \ref{footnote1}).
}

Since for all $k\geq 2$ we have $k!<k^k$ and thus $\frac{\log(k!)}{k\log(k)} < 1$, we deduce that the Hausdorff dimension of $S$ is strictly less than $\delta$.
\end{proof}

\ignore{
{\bf{Make a note of assumptions}}
By the BEST theorem, the number of de Bruijn sequences of order $n$ in $\alphabet$ is equal to $(k - 1)!^{\#(V(X))}$ times the number of spanning trees of $X$, where $X = \dBG_{n - 1}$ (so that $\#(V(X)) = k^{n - 1}$). The number of spanning trees is at most $k^{\#(V(X))}$, since an edge must be selected emanating from each vertex, and each vertex has $k$ edges. So there are at most $k!^{\#(V(X))}$ de Bruijn sequences. Since these sequences are of length $\sim k^n = k\#(V(X))$, the lower box dimension is at most
\[
\frac{\log(k!^{\#(V(X))})}{k\#(V(X))\log(\base)} = \frac{\log(k!)}{k\log(k)}\delta
\]
since $\delta = \log(k) / \log(\base)$, (see Footnote \ref{footnote1}).

Since for all $k\geq 2$ we have
\[
k! = \prod_{i = 1}^k i < \prod_{i = 1}^k k = k^k
\]
and thus $\frac{\log(k!)}{k\log(k)} < 1$, we deduce that the Hausdorff dimension of the set of de Bruijn sequences is strictly less than $\delta$.
}

\subsection{The lower bound}

The proof of the lower bound is significantly more involved, and will require a few preliminary results. We begin with the following proposition:

\begin{proposition}
\label{propositionmetricstructure}
Let $X$ be a $k$-regular connected directed graph, fix $x_0\in V(X)$, and let $\EE$ be the set of Eulerian \paths\ of $X$ that start and end at $x_0$. Then there exists $\EE'\subset \EE$ such that:
\begin{itemize}
\item[(i)] $\#(\EE') = k\cdot (k - 1)!^{\#(V(X))}$;
\item[(ii)] If $\delta$ is a path of length $\ell_\delta$ starting at $x_0$, then the number of \paths\ in $\EE'$ that extend $\delta$ is at most $k\cdot (k - 1)!^{\#(V(X)) - \ell_\delta/k}$.
\end{itemize}
\end{proposition}
\begin{proof}
Since $X$ is connected, by Lemma \ref{lemmaexistsT} there exists a directed spanning tree $T$ rooted at $x_0$. Let $\EE'$ be the set of Eulerian \paths\ $\delta$ that start and end at $x_0$ such that for all $xy\in E(X)$ and $xz\in E(T)$ with $y\neq z$, the edge $xy$ appears in $\delta$ before $xz$ does. Equivalently, $\EE' = \{f(T,\oo) : \oo\in O(T)\}$ where the notation is as in the proof of the BEST theorem. Then the proof of the BEST theorem implies that $\#(\EE') = \#(O(T)) = k\cdot (k - 1)!^{\#(V(X))}$. Now let $\delta$ be a path starting at $x_0$ that has at least one extension in $\EE'$. For each $\oo\in O(T)$, the \path\ $f(T,\oo)$ is an extension of $\delta$ if and only if the algorithm described in the proof of the BEST theorem produces $\delta$ on input $\oo$. Equivalently, $f(T,\oo)$ is an extension of $\delta$ if for each edge $xy$ of $\delta$, the rank of $xy$ according to $o_x$ is the same as its rank according to its location in $\delta$. The number of elements $\oo\in O(T)$ satisfying this condition is given by the formula
\begin{align*}
N_\delta = &\prod_{x\in V(X)} [\#(E_x \butnot (E(\delta)\cup E(T))]!\\
= \#(E_{x_0}\butnot E(\delta)) \,\cdot &\prod_{x\in V(X)} [\#(E_x\butnot E(\delta)) - 1]! \leq k\cdot \prod_{x\in V(X)} [\#(E_x\butnot E(\delta)) - 1]!
\end{align*}
where $E_x$ denotes the set of edges with initial vertex $x$, and $E(\delta)$ denotes the edge set of $\delta$. Here we use the convention $(-1)! = 1$, since if $E_x\butnot E(\delta) = \emptyset$, then there is exactly one ordering $o_x$ satisfying the appropriate condition, namely the ordering determined by $\delta$, and by hypothesis the element $v_x$ comes last in this ordering. Now since
\[
(i - 1)! \leq (k - 1)!^{i/k} \all i = 0,\ldots,k,
\]
we have
\[
N_\delta \leq k\cdot (k - 1)!^{M/k},
\]
where
\[
M \df \sum_{x\in V(X)} \#(E_x\butnot E(\delta)) = \#(E(X)\butnot E(\delta)) = k\#(V(X)) - \ell_\delta.
\qedhere\]
\end{proof}

The next result will furnish the lower bound for $k \geq 4$. Although it is valid for $k=3$, it provides no useful information in this case since $0$ is always a (trivial) lower bound on the dimension.

\begin{corollary}
\label{lowerbound4}
Let the notation be as in Theorem \ref{MainTheoremDeBruijn}, and let $S$ be the set of numbers in $F$ with totally de Bruijn base $\base$ expansions. Assume that $k \geq 4$. Then the Hausdorff dimension of $S$ is bounded below by $\alpha_k \delta > 0$, where $\delta$ is the Hausdorff dimension of $F$ (and equals $\log(k)/\log(\base)$), and
\begin{equation}
\label{alphak4}
\alpha_k = \frac{\log(k - 2)!}{k \log(k)}\cdot
\end{equation}
\end{corollary}

\ignore{
\begin{corollary}
\label{dimension3}
Fix $k\geq 3$. Then the Hausdorff dimension of the set of de Bruijn sequences in $k$ symbols is at least
\[
\alpha = \frac{\log (k - 2)!}{k\log(k)}\cdot
\]
\end{corollary}
}

Before we turn to the proof, we recall the so-called \textsl{Mass Distribution Principle}, an extremely useful tool for bounding the Hausdorff dimension from below.
\begin{lemma}[{\cite[Principle 4.2]{Falconer_book}}]
\label{lemmamdp}
Let $F$ be a metric space, and let $\mu$ be a measure on $F$ such that $0<\mu (F)<\infty$. Fix $s,\epsilon > 0$, and suppose that there exists $C > 0$ such that $\mu (\set)\leq C\cdot\diam (\set) ^s$ for every set $\set \subset F$ such that $\diam (\set)\leq \epsilon$. Then
\[
\HD (F) \geq s.
\]
\end{lemma}

\begin{proof}[Proof of Corollary \ref{lowerbound4}]
Fix $n \in \N$, and let $\omega = \omega_1\cdots \omega_{k^n + n - 1}$ be a de Bruijn sequence of order $n$ in $\alphabet$. Since the path induced by $\omega$ on $\dBG_{n - 1}(A)$ is an Eulerian path in a directed graph in which each vertex has equal in-degree and out-degree, it must start and end at the same vertex, which means that the first $(n - 1)$ letters of $\omega$ are the same as the last $(n - 1)$ letters, i.e. $\omega_{k^n + i} = \omega_i$ for all $i = 1,\ldots,n - 1$.\Footnote{This phenomenon is related to the fact that we consider non-cyclic de Bruijn sequences instead of cyclic ones: each cyclic de Bruijn sequence $\omega = \omega_1\cdots\omega_{k^n}$ corresponds to a non-cyclic de Bruijn sequence $\omega_1\cdots \omega_{k^n} \omega_1 \cdots \omega_{n - 1}$ that is longer but has the same number of consecutive substrings. This correspondence makes it obvious that the first $(n - 1)$ letters of a non-cyclic de Bruijn word are expected to be the same as the last $(n - 1)$ letters. However, by itself this is not a proof, because our definition of non-cyclic de Bruijn sequences did not assume that they were constructed from cyclic ones.} Now let $\omega_{k^n + n} = \omega_n$ and $\omega' = \omega_1\cdots \omega_{k^n + n}$. Then the first $n$ letters of $\omega'$ are the same as the last $n$ letters, but no other block of $n$ letters is repeated in $\omega'$.

Let $\dBG = \dBG_n(A)$ be the de Bruijn graph of order $n$ on $\alphabet$, and let $\gamma = \gamma_1\cdots \gamma_{k^n + 1}$ be the path induced by $\omega'$ on $\dBG$. Then $\gamma$ is a Hamiltonian cycle (i.e. a simple path traversing each vertex once). The collection of de Bruijn sequences of order $(n + 1)$ that extend $\omega'$ is isomorphic to the collection of Eulerian \paths\ on $\dBG$ that extend $\gamma$.

Let $x_0 \df \gamma_1 = \gamma_{k^n + 1}$ be the common initial and terminal vertex of $\gamma$. Then the collection of Eulerian \paths\ of $\dBG$ that extend $\gamma$ is isomorphic to the set of Eulerian \paths\ of $X_\omega \df \dBG \butnot E(\gamma)$ that start and end at $x_0$, which we denote by $\EE(\omega)$. Since $X_\omega$ is a $(k - 1)$-regular connected directed graph whose vertex set has size $k^n$ (see the proof of \cite[Lemma 3]{BecherHeiber} for connectedness), we may use Proposition \ref{propositionmetricstructure} to extract a subset $\EE'(\omega) \subset \EE(\omega)$. Pulling this subset back via the appropriate correspondences gives us a set $S'(\omega)$, contained in the set of all de Bruijn sequences of order $(n + 1)$ extending $\omega'$ (and thus also extending $\omega$), with the following properties:
\begin{itemize}
\item[(i)] $\#(S'(\omega)) = (k - 1)\cdot (k - 2)!^{k^n}$.
\item[(ii)] If $\tau$ is a sequence of length $\ell_\tau$ extending $\omega$, then the number of sequences in $S'(\omega)$ that extend $\tau$ is at most $(k - 1)\cdot (k - 2)!^{k^n - (\ell_\tau - \ell_\omega - 1)/k}$, where $\ell_\omega = k^n + n - 1$ is the length of $\omega$.
\end{itemize}
Now we proceed to define a probability measure $\mu$ on $F \equiv E^\N$ via a random algorithm: start with a fixed de Bruijn sequence $\omega^{(1)}$ of order $1$, and if $\omega^{(n)}$ is a de Bruijn sequence of order $n$, then let $\omega^{(n + 1)} \in S'(\omega^{(n)})$ be chosen randomly with respect to the uniform measure on $S'(\omega^{(n)})$, independent of all previous selections. Let $\omega$ be the unique infinite sequence that extends all of the finite sequences $\omega^{(n)}$ ($n\in\N$). Then $\omega$ is a base $\base$ expansion of a unique point $\pi(\omega)\in F$. (The point $\pi(\omega)$ may have a base $\base$ expansion other than $\omega$, but there is no other point with base $\base$ expansion $\omega$.) We let $\mu$ be the probability measure describing the distribution of the random variable $\pi(\omega)$. (The existence of such a $\mu$ can be guaranteed e.g. by the Kolmogorov extension theorem.)

To demonstrate that $\mu$ satisfies the hypotheses of the mass distribution principle, we first estimate the measure of cylinder sets of a certain length, then arbitrary cylinder sets, then balls. Here a \emph{cylinder set} is a set of the form $[\tau] = \{\pi(\omega) : \omega_i = \tau_i \all i = 1,\ldots,\ell_\tau\}$, where $\tau = \tau_1\cdots\tau_{\ell_\tau}$ is a finite sequence in the alphabet $\alphabet$. Our first estimate is easy: if $\ell_\tau = k^{n + 1} + n$ for some $n$, then $[\tau]$ is precisely the set of $\pi(\omega)$ in the above construction such that $\omega^{(n + 1)} = \tau$, so $\mu([\tau])$ is just the probability that $\omega^{(n + 1)} = \tau$, i.e.
\begin{equation}
\label{MDP1}
\mu([\tau]) = \prod_{i = 1}^n \frac{1}{\#(S'(\tau^{(i)}))} = \prod_{i = 1}^n \frac{1}{(k - 1)\cdot (k - 2)!^{k^i}} \leq (k - 2)!^{-(k^n + k^{n - 1} + n/k)}
\end{equation}
if it is possible that $\omega^{(n + 1)} = \tau$, and $\mu([\tau]) = 0$ otherwise. Now consider the more general case where the length of $\tau$ satisfies $k^n + n - 1 < \ell_\tau \leq k^{n + 1} + n$ for some $n$. Then by (ii) above, $[\tau]$ contains at most $(k - 1)\cdot(k - 2)!^{k^n - (\ell_\tau - (k^n + n))/k}$ cylinders of length $k^{n + 1} + n$. Combining with \eqref{MDP1} shows that
\begin{align*}
\mu([\tau]) &\leq (k - 1)\cdot \exp_{(k-2)!}\big(k^n - (\ell_\tau - (k^n + n))/k) - (k^n + k^{n-1} + n/k)\big)\\
&= (k - 1) \cdot (k-2)!^{-\ell_\tau/k}.
\end{align*}
Here and hereafter we use the notation $\exp_x(y) \df x^y$.

To apply the mass distribution principle (Lemma \ref{lemmamdp}), we now need to relate this measure to the diameter of the cylinder $[\tau]$. Since elements of $[\tau]$ have the first $\ell_\tau$ digits of their base $\base$ expansions fixed, the diameter of $[\tau]$ is approximately $\base^{-\ell_\tau}$ (to be precise, it is $c\cdot \base^{-\ell_\tau}$ for some constant $0 < c \leq 1$). Thus
\[
\diam([\tau])^{\alpha_k \delta} = c^{\alpha_k \delta} \exp_\base\left(-\ell_\tau \frac{\log(k - 2)!}{k\log(k)} \frac{\log(k)}{\log(\base)}\right) = c^{\alpha_k \delta} \cdot (k - 2)!^{-\ell_\tau/k},
\]
so
\[
\mu([\tau]) \leq C \cdot \diam([\tau])^s,
\]
where $C = (k - 1)\cdot c^{-\alpha_k \delta}$ and $s = \alpha_k \delta$. But any subset of $F$ can be covered by at most two cylinder sets with comparable diameter, so a similar formula holds for arbitrary sets. Thus by Lemma \ref{lemmamdp}, we have $\HD(S) \geq s = \alpha_k \delta$.
\end{proof}

As is evident from Corollary \ref{lowerbound4}, we now have to deal with the cases $k=2$ and $k=3$ separately, since in those cases the formula \eqref{alphak4} gives $\alpha_2 = \alpha_3 = 0$, which is not a useful bound. Note that the Cantor ternary set falls into the case $k = 2$, since its set of admissible numerators is $\alphabet = \{0,2\}$.

\begin{proposition}
\label{propositioncountA2}
If $k = 2$ and $\omega$ is a de Bruijn sequence of order $(n - 2)$ in $\alphabet$, then the number of de Bruijn sequences of order $(n + 1)$ that extend $\omega$ is at least $2^{2^{n - 2}}$.

In the case where $k = 3$ and $\omega$ is a de Bruijn sequence of order $(n - 1)$ in $\alphabet$, then the number of de Bruijn sequences of order $(n + 1)$ that extend $\omega$ is at least $4^{3^{n - 1}}$.
\end{proposition}


\begin{proof}
For convenience we let $\Delta = 2$ if $k = 3$, and $\Delta = 3$ if $k = 2$; then $\omega$ is a de Bruijn sequence of order $(n - \Delta + 1)$. The first paragraph of Corollary \ref{lowerbound4} shows that the first $(n - \Delta)$ letters of $\omega$ are the same as the last $(n - \Delta)$ letters. So if we extend $\omega$ to a word $\omega'$ of length $k^{n - \Delta + 1} + n$ by letting $\omega_{k^{n - \Delta + 1} + i} = \omega_i$ for $i = n - \Delta + 1,\ldots,n$, then the first $n$ letters of $\omega'$ are the same as the last $n$ letters, but no other block of $n$ letters is repeated.

Let $\dBG$ be the de Bruijn graph of order $n$ on $\alphabet$, and let $\gamma$ be the path induced by $\omega'$ on $\dBG$. The length of $\gamma$ is $\ell_\gamma = k^{n - \Delta + 1}$, and $\gamma$ is a simple path that starts and ends at the same vertex $x_0$. As in the proof of Corollary \ref{lowerbound4}, we let $X = X_\omega = \dBG\butnot E(\gamma)$, where $E(\gamma)$ is the edge set of $\gamma$. The collection of de Bruijn sequences of order $(n + 1)$ that extend $\omega$ is isomorphic to the collection of Eulerian \paths\ on $\dBG$ that extend $\gamma$, which in turn is isomorphic to the collection of Eulerian paths on $X_\omega$ that start and end at $x_0$. By the BEST theorem, the cardinality of this collection is
\[
\NN \df \#(\TT) \cdot \deg(x_0;X_\omega)\cdot \prod_{x\in V(\dBG)} [\deg(x;X_\omega) - 1]!
\]
If $k = 3$, we complete the proof with the following calculation:
\begin{align*}
\NN &\geq \prod_{x\in V(\dBG)} [\deg(x;X_\omega) - 1]!
= \exp_2(\#\{x\in V(\dBG) : \deg(x;X_\omega) = 3\})\\
&= \exp_2(\#(V(\dBG)) - \ell_\gamma) = \exp_2(3^n - 3^{n - 1}) = 4^{3^{n - 1}}.
\end{align*}
In the first inequality, we have used Lemma \ref{lemmaexistsT} and the proof of \cite[Lemma 3]{BecherHeiber} to deduce that $\#(\TT) \geq 1$.

For the remainder of the proof, we assume that $k = 2$. In this case, the strategy of the above calculation cannot work, since we have $[\deg(x;X_\omega) - 1]! = 1$ for all $x\in V(\dBG)$ and thus $N \leq 2\#(\TT)$. Instead we must estimate the number of spanning trees in $X_\omega$.

Let $S$ be the set of sequences of length $(n - 1)$ that do not occur in $\omega$, and note that $\#(S) = 2^{n - 1} - 2^{n - 2} = 2^{n - 2}$. For each $\tau\in S$, let $E_\tau = \{a\tau b : a,b\in \alphabet\} \subset E(X_\omega)$, where $a\tau b$ is shorthand for $(a\tau)(\tau b)$, the edge from the vertex $a\tau$ to the vertex $\tau b$. Note that the sets $E_\tau$ ($\tau\in S$) are disjoint.

\begin{lemma}
\label{lemmaTprime}
If $T$ is a directed spanning tree and $\tau\in S$, then there exists a directed spanning tree $T' \neq T$ such that $T'\butnot E_\tau = T\butnot E_\tau$.
\end{lemma}
\begin{subproof}
By contradiction, suppose that the conclusion of the lemma is false, i.e. that there exists no such spanning tree $T'$.

Denote the partial order on $V(\dBG)$ induced by the tree $T$ by $<$, i.e. write $x < y$ if there is a path in $T$ from $x$ to $y$, and write $x \leq y$ if either $x < y$ or $x = y$. We write $x <^* y$ if $x$ is a direct descendant of $y$, i.e. if $xy\in E(T)$. For each $a\in \alphabet$, let $f(a)\in \alphabet$ be chosen to satisfy $a\tau f(a)\in E(T)$, and let $g(a) = \sigma(f(a))$, where $\sigma:\alphabet \to \alphabet$ is the permutation that swaps the two elements of $\alphabet$. Consider the graph $T' = T\cup\{a\tau g(a)\}\butnot\{a\tau f(a)\}$. Then $T' \neq T$ and $T'\butnot E_\tau = T\butnot E_\tau$, so by the contradiction hypothesis, $T'$ is not a directed spanning tree, which implies that $\tau g(a) \leq a\tau$. On the other hand, we have $a\tau <^* \tau f(a)$ since $a\tau f(a)\in T$. Now write $\alphabet = \{a,b\}$, $c = f(a)$, and $d = \sigma(c) = g(a)$. Then either $f(b) = c$ or $f(b) = d$, and thus we have one of the following two diagrams:
\[
\tau d \leq a\tau <^* \tau c >^* b\tau \geq \tau d \;\;\text{ or }\;\;
\tau d \leq a\tau < \tau c \leq b\tau < \tau d.
\]
Both diagrams are impossible for directed trees: the left-hand diagram is impossible because if $a\tau$ and $b\tau$ are siblings, then they have no common descendants, while the right-hand diagram is disjoint because it is a nontrivial directed loop. This is the desired contradiction.
\end{subproof}

It follows from Lemma \ref{lemmaTprime} that there exists a function $\phi:\TT\times S \to \TT$ such that for all $T\in \TT$ and $\tau \in S$, we have $\phi(T,\tau) \neq T$ and $\phi(T,\tau) \butnot E_\tau = T\butnot E_\tau$.

Now by Lemma \ref{lemmaexistsT} and the proof of \cite[Lemma 5]{BecherHeiber}, $X$ has a directed spanning tree $T_0$ rooted at $x_0$. Let $(\tau_i)_{i = 1}^N$ be an indexing of $S$, where $N = 2^{n - 2}$. Given $\omega\in \{0,1\}^N$, we define recursively
\begin{align*}
T_{\omega,0} &= T_0,&
T_{\omega,i} &= \begin{cases}
T_{\omega,i - 1} & \omega_i = 0\\
\phi(T_{\omega,i - 1},\tau_i) & \omega_i = 1
\end{cases}.
\end{align*}
Then the map $\{0,1\}^N\ni \omega\to T_{\omega,N} \in \TT$ is injective. Thus $\NN \geq \#(\TT) \geq \#(\{0,1\}^N) = 2^{2^{n - 2}}$, which completes the proof.
\end{proof}

\begin{corollary}
\label{lowerbound23}
Let the notation be as in Theorem \ref{MainTheoremDeBruijn}. Suppose that $k \leq 3$, and let
\begin{align*}
\alpha_k &= \begin{cases}
1/49 & \text{if } k = 2\\
(8\cdot (9\log_4(3) - 1))^{-1} & \text{if } k = 3
\end{cases}&
\Delta &= \begin{cases}
3 & \text{if } k = 2\\
2 & \text{if } k = 3
\end{cases}
\end{align*}
Then the Hausdorff dimension of the set
\[
\{\pi(\omega) \in F : \text{$B_\omega$ contains an arithmetic progression with gap size $\Delta$}\}
\]
is at least $\alpha_k \delta$.
\end{corollary}

\begin{proof}
Let $B = 2$ if $k = 2$ and $B = 4$ if $k = 3$. Then $\alpha_k = ((k^\Delta - 1)\cdot (k^\Delta \log_B(k) - 1))^{-1}$, and Proposition \ref{propositioncountA2} can be expressed uniformly as follows: If $\omega$ is a de Bruijn sequence of order $n$ in $\alphabet$, then the number of de Bruijn sequences of order $n + \Delta$ that extend $\omega$ is at least $\exp_B(k^n)$. We denote the set of all such extensions by $S'(\omega)$.

As in the proof of Corollary \ref{lowerbound4}, we define a probability measure $\mu$ by a random algorithm: let $\omega^{(1)}$ be a fixed de Bruijn sequence of order $\Delta$, and if $\omega^{(n)}$ is a de Bruijn sequence of order $n\Delta$, then let $\omega^{(n + 1)}$ be chosen randomly with respect to the uniform measure on $S'(\omega^{(n)})$, independent of all previous selections. As before we let $\omega\in \alphabet^\N$ be the unique common extension, we let $\pi(\omega)\in F$ be the unique number for which $\omega$ is a base $\base$ expansion, and we let $\mu$ be the probability measure describing the distribution of $\pi(\omega)$.

As before, we first estimate the measure of special cylinders, then arbitrary cylinders, then balls. For ease of notation we fix $k = 3$ in this proof; for the case $k = 2$ one can apply the substitutions $9 \mapsto 8$, $8\mapsto 7$, $4\mapsto 2$, $3\mapsto 2$, and $2\mapsto 3$. Fix $n\in\N$ and let $\tau$ be a sequence of length $9^n + 2n - 1$ in $\alphabet$. Then
\[
\mu([\tau]) \leq \prod_{i = 1}^{n - 1} \frac{1}{\#(S'(\tau^{(i)}))} \leq \prod_{i = 1}^{n - 1} \frac{1}{\exp_4(3^{2i})} = \exp_4\left(-\frac{9^n - 9}{9 - 1}\right).
\]
Now let $\tau$ be an arbitrary sequence of length $9^n + 2n - 1 < \ell_\tau \leq 9^{n + 1} + 2(n + 1) - 1$. There are two ways that we could bound $\mu([\tau])$:
\begin{itemize}
\item[1.] Since $[\tau] \subset [\tau^{(n)}]$, we have
\[
\mu([\tau]) \leq \mu([\tau^{(n)}]) \leq \exp_4\left(-\frac{9^n - 9}{8}\right).
\]
\item[2.] Since $[\tau]$ can be written as the union of at most $\exp_3(9^{n + 1} + 2(n + 1) - 1 - \ell_\tau)$ cylinder sets corresponding to de Bruijn sequences of order $2(n + 1)$, we have
\[
\mu([\tau]) \leq \exp_3(9^{n + 1} + 2(n + 1) - 1 - \ell_\tau) \cdot \exp_4\left(-\frac{9^{n + 1} - 9}{8}\right).
\]
\end{itemize}

Which of these bounds is better depends on the value of $\ell_\tau$. Now, as in the proof of Corollary \ref{lowerbound4}, we have $\diam([\tau]) = c\cdot \base^{-\ell_\tau}$ for some constant $c$. Fix $0 < s < \alpha_3 \delta$. To apply the mass distribution principle, we need to show that
\[
\mu([\tau]) \leq C \cdot \diam([\tau])^s
\]
for some constant $C$. It is enough to show that
\[
\textbf{min}\left(4^{-9^n/8},
\exp_3(9^{n + 1} + 2n - \ell_\tau) \cdot 4^{-9^{-n + 1}/8}\right) \leq C \cdot \base^{-s\ell_\tau} = C \cdot 3^{-t\ell_\tau},
\]
possibly with a different value of $C$, where $t = s/\delta < \alpha_3 < 1$. Equivalently, we need to show that
\[
\textbf{min}\left(4^{-9^n/8} \cdot 3^{t\ell_\tau},
3^{9^{n + 1}} \cdot 9^n \cdot 4^{-9^{n + 1}/8}\cdot 3^{(t - 1)\ell_\tau}\right) \leq C.
\]
Now the first input to the binary operator $\textbf{min}$ is an increasing function of $\ell_\tau$, while the second input is a decreasing function of $\ell_\tau$. It follows that the largest value the left-hand side can attain is the value attained when the two inputs to $\textbf{min}$ are equal, i.e. when
\[
4^{-9^n/8} = 3^{9^{n + 1}} \cdot 9^n \cdot 4^{-9^{n + 1}/8}\cdot 3^{-\ell_\tau},
\]
at which point the left-hand side is
\[
4^{-9^n/8} \cdot \left(3^{9^{n + 1}} \cdot 9^n \cdot 4^{9^n/8 - 9^{n + 1}/8}\right)^t.
\]
We need this expression to be bounded as $n\to\infty$. Applying the change of variables $x = 9^n$, we need to show that
\[
\limsup_{x\to\infty} 4^{-x/8} \cdot \left(3^{9x}\cdot x\cdot 4^{x/8 - 9x/8}\right)^t < \infty.
\]
This is true if and only if
\[
4^{-1/8} \cdot (3^9 \cdot 4^{-1})^t < 1,
\]
which in turn is true if and only if $t < \alpha_3$. This proves that the hypothesis of the mass distribution principle holds for cylinder sets. As in the proof of Corollary \ref{lowerbound4}, any subset of $F$ can be covered by at most two cylinder sets with comparable diameter, so the hypothesis of the mass distribution principle holds for arbitrary sets as well.
\end{proof}

Combining Corollaries \ref{lowerbound4} and \ref{lowerbound23} yields Theorem \ref{MainTheoremDeBruijn}.

\begin{remark}
\label{remarkk3k2}
Either of the strategies used in this proof, the (simpler) strategy for the $k = 3$ case or the (more complicated) strategy for the $k = 2$ case, could have been used (after minor modification) in the case $k \geq 4$ as well, but the resulting bound would have been significantly worse, measured by the fact that the analogues of $\alpha_k$ would not have tended to $1$. Similarly, the strategy for the $k = 2$ case could have been used for the $k = 3$ case, again resulting in a worse bound. In general, the principle is that whatever techniques work for one value of $k$ will also work for higher values of $k$, but may not give very good estimates for higher values of $k$.
\end{remark}

\section{Intrinsic Diophantine approximation}
\label{DiophantineApproximation}

\subsection{Diophantine approximation -- a brief survey}

We first recall some definitions and state some well-known classical theorems:

\ignore{

\begin{theorem*}[Dirichlet's Approximation Theorem]
For each $x \in \R$ and for any $Q \in \mathbb{N}$ there exists $p/q\in \Q$ with $1\leq q \leq Q$, such that
\[
\bigl|x-p/q\bigr|<\frac{1}{qQ}.
\]
\end{theorem*}

}

\begin{definition}
\label{Height}
Let $H:\Q \to\R _{>0}$ be a function. We think of $H$ as a ``height function'', and for all $p\in\Z$ and $q\in\N$, we define the {\emph{height}} of $p/q$ to be the number $H(p/q)$. We say that a function $\psi :\R_{>0}\to\R _{>0}$ is a {\emph{Dirichlet function}} (with respect to the height function $H$) if for every $x\in\R\setminus\Q$ there exist infinitely many rationals $p/q$ such that
\[
\label{dirich}
\bigl|x-p/q\bigr|<\psi(H(p/q)).
\]
\comdavid{In our other papers we introduce a constant factor... maybe we should try to be consistent?}
\end{definition}

Historically speaking, the only height function considered on the unit interval $[0,1]$ was the function $H_\std(p/q)=q$, where $p$ and $q$ are chosen in reduced form, i.e. $\gcd(p,q) = 1$. We will refer to this as the \emph{standard} height. It is readily verified that for example $\psi_0(q)=1$ and $\psi_1(q)=1/q$ are Dirichlet functions with respect to the standard height function and using the terminology of Definition \ref{Height}, Dirichlet's approximation theorem may be stated as follows:

\begin{theorem*}[Dirichlet]
\label{Dirichlet}
$\psi_2(q)=\frac{1}{q^2}$ is a Dirichlet function with respect to the standard height function.\Footnote{In fact, Dirichlet's theorem furnishes a similar result for all dimensions $d$. It was recently pointed out to us by Y. Bugeaud that the one-dimensional version of this result is actually much older, coming directly from the theory of continued fractions (see e.g. \cite[displayed equation on p.28]{Legendre}). Nevertheless, we call the theorem ``Dirichlet's theorem'' so as to conform to usual practice.}
\end{theorem*}

For our purposes, although of interest in its own right, an improvement of a Dirichlet function by a multiplicative constant is not significant. More precisely:

\begin{definition}
\label{definitionoptimal}
We say that a Dirichlet function $\psi$ is {\emph{optimal}} if there does not exist a Dirichlet function $\phi$ for which $\lim_{q\to\infty}\frac{\phi(q)}{\psi(q)}\to 0$.
\end{definition}

It is clear that Dirichlet's theorem implies that the Dirichlet functions $\psi_0$ and $\psi_1$ defined above are not optimal. The optimality of the function $\psi_2(q)=1/q^2$ was demonstrated by Liouville, who proved that quadratic irrationals are badly approximable. A real number $x$ is called \textsl{badly approximable} if there exists $c(x)>0$ such that
\[
\bigl|x- p/q \bigr| > \frac{c(x)}{q^2} \hspace{2mm} \text{ for all } p/q \in \Q.
\]
Liouville's result was later significantly improved by Jarn\'ik, who proved that the Hausdorff dimension of the set of badly approximable numbers is $1$.

\ignore{

An important observation regarding Dirichlet's theorem is that no a priory restriction is given on the reducibility of the approximating rationals, while by the definition of badly approximable numbers we have to assume that the rationals
are in \textsl{reduced form}, i.e. the ``worst case scenario'' for the inequality to be satisfied.\\

}

\subsection{Iterated function systems, limit sets, and Hausdorff dimension}
\label{Hutchinson}

Let $k\geq 2$ be an integer. In what follows, we shall consider a finite famiily $(S_i)_{i = 1}^k$ of contracting similarities on the unit interval $I=[0,1]$. This means that for every $1 \leq i \leq k$, the map $S_i : I \to I$ satisfies
\[
|S_i(x)-S_i(y)|=c_i|x-y| \all x,y\in I
\]
for some $0<c_i<1$. We shall call such a family of similarities an \textsl{Iterated Function System} or IFS. A nonempty compact set $F \subset I$ is said to be the \textsl{attractor} or the \textsl{limit set} of the IFS if
\[
F=\bigcup_{i=1}^{k}S_i(F).
\]
It is well known (see e.g., \cite[Chapter 9]{Falconer_book}) that the attractor $F$ exists and is unique. Furthermore, if there exists a bounded nonempty open set $U$ such that
\[
\bigcup _{i=1}^{k}S_i(U)\subset U
\]
with the union disjoint, then the IFS is said to satisfy the \emph{open set condition}. In this case, the Hausdorff dimension of the attractor is equal to the unique solution $s > 0$ of the equation
\begin{equation}
\label{hutchinson}
\sum_{i=1}^{k}c_i^{s}=1.
\end{equation}
We say that that the IFS $(S_i)_{i = 1}^k$ satisfies the \textsl{strong separation condition} if
\[
S_i(F)\cap S_j(F) = \emptyset
\]
for all $i\neq j$, where $F$ is the attractor.\Footnote{Note that the strong separation condition implies (but is not implied by) the open set condition.}\\

A particularly important example of an iterated function system is the system
\begin{equation}
\label{baryIFS}
S_i(x) = \frac{i + x}{\base}, \;\;\;\; i \in C(\base) \df \{0,\ldots,\base - 1\},
\end{equation}
where $\base \geq 2$ is fixed. This system satisfies the open set condition (with $U = (0,1)$) but not the strong separation condition, and its attractor is the entire interval $I$. In some sense this IFS encodes the base $\base$ expansion(s) of any number in the interval $[0,1]$, since the number
\[
x = \pi(\omega) = 0.\omega_1\omega_2\cdots \text{ (base $\base$) } = \sum_{i = 1}^\infty \frac{\omega_i}{\base^i}
\]
can be written as
\[
x = \lim_{n\to\infty} S_{\omega_1}\circ\cdots\circ S_{\omega_n}(0).
\]
By looking at subsystems of the system \eqref{baryIFS}, we can find IFSes whose limit sets can be described in terms of base $\base$ expansions. Fix $\alphabet \subset C(\base)$, and consider the subsystem of \eqref{baryIFS} consisting of the similarities $(S_i)_{i\in \alphabet}$. We call such a subsystem a \emph{base $\base$ IFS}. Its limit set is precisely the set of all numbers in $[0,1]$ that have at least one base $\base$ expansion whose digits all lie in $\alphabet$, i.e.
\begin{equation}
\label{equationfractal}
F = \left\{x \in [0,1] : \exists \omega \in \alphabet^\N \text{ with } x = \sum^\infty_{i=1} \frac{\omega_i}{\base^i}\right\}.
\end{equation}
For example, if $\base = 3$ and $\alphabet = \{0,2\}$, then $F$ is the standard Cantor ternary set, i.e. the set of all numbers in $[0,1]$ that have at least one base $3$ expansion containing only the digits $0$ and $2$.

It follows directly from \eqref{hutchinson} that the Hausdorff dimension of the base $\base$ IFS corresponding to an alphabet $\alphabet \subset C(\base)$ is precisely $\log\#(\alphabet) / \log(\base)$.

We remark that it is easy to check whether a base $\base$ IFS satisfies the strong separation condition:
\begin{observation}
The base $\base$ IFS defined by the alphabet $\alphabet \subset C(\base)$ satisfies the strong separation condition if and only if at least one of the following is true:
\begin{itemize}
\item[(1)] $0 \notin \alphabet$.
\item[(2)] $\base - 1 \notin \alphabet$.
\item[(3)] $\alphabet$ does not contain any pair of consecutive integers.
\end{itemize}
\end{observation}
If a base $\base$ IFS satisfies the strong separation condition, then every element of its limit set $F$ has exactly one base $\base$ expansion whose digits come from $\alphabet$. In this case, there is no ambiguity about talking about ``the base $\base$ expansion'' of a number in $F$, since we understand that if there is more than one base $\base$ expansion, then we are talking about the one whose digits come from $\alphabet$.

\subsection{Intrinsic approximation on limit sets}
\label{fractals}

Let $F\subset\R$ be a closed set, which we will think of as a fractal. The field of \emph{intrinsic Diophantine approximation} is concerned with finding rational approximations to an irrational number $x\in F$ by rational numbers that lie \emph{on} the fractal $F$. Thus Mahler's first question is about intrinsic approximation on the Cantor set.  More generally, one may ask about intrinsic approximation on the attractor of any similarity IFS. This leads to the following definition:

\begin{definition}
\label{Height2}
Let $F \subset \R$ be a closed set, and let $H:F\cap \Q \to\R _{>0}$ be a height function. We say that a function $\psi :\R_{>0}\to\R _{>0}$ is an {\emph{intrinsic Dirichlet function on $F$}} (with respect to the height function $H$) if for every $x\in F\setminus\Q$ there exist infinitely many rationals $p/q \in F\cap \Q$ such that
\[
\label{dirich2}
\bigl|x-p/q\bigr|<\psi(H(p/q)).
\]
Optimality of intrinsic Dirichlet functions can be defined in the same way as in Definition \ref{definitionoptimal}.
\end{definition}

We have the following result:

\begin{proposition}[{\cite[Corollary 2.2]{BFR}}]
\label{BFRresult}
Let $F$ be the limit set of a base $\base$ IFS, and let $\delta$ be the Hausdorff dimension of $F$. Then for all $x\in F$, there exist infinitely many rational numbers $p/q\in F$ ($p\in\mathbb{Z}$, $q\in\mathbb{N}$) such that
\[
\bigl|x-p/q\bigr| < \frac{1}{q(\log_\base q)^{1/\delta}}\cdot
\]
In other words, the function $\psi_*(q) = (q\cdot (\log_\base q)^{1/\delta})^{-1}$ is an intrinsic Dirichlet function on $F$ for the standard height function.
\end{proposition}

\section{The symbolic height function}

\label{Diophantine}

Let $F$ be the limit set of a base $\base$ IFS satisfying the strong separation condition, and fix a rational number $r \in F \cap \Q$. It is well known that the base $\base$ expansion of $r$ is preperiodic, i.e.
\begin{equation}
\label{equationrationalform} r = 0.\omega_1\ldots \omega_i \overline{\omega_{i+1}\ldots \omega_{i + j} \vphantom{t}} \;\;\text{ (base $\base$)},
\end{equation}
for some $i \geq 0$, $j \geq 1$, and $\omega_1,\ldots,\omega_{i + j}\in \alphabet$. Here the bar indicates that the string $\omega_{i + 1} \cdots \omega_{i + j}$ is infinitely repeated. Rewriting the right-hand side as a sum of fractions yields
\begin{eqnarray*}
r &=& \frac{\omega_1 \ldots \omega_i}{\base^i} +  \sum^\infty_{m=1} \frac{\omega_{i+1}\ldots \omega_{i + j}}{\base^{i + mj}} \\ &=& \frac{\omega_1 \ldots \omega_i}{\base^i} + \frac{\omega_{i+1}\ldots \omega_{i + j}}{\base^{i}}\cdot \frac{1/\base^{j}}{1 - 1/\base^{j}}\\ &=& \frac{\omega_1 \ldots \omega_i}{\base^i} + \frac{\omega_{i+1}\ldots \omega_{i + j}}{\base^{i}}\cdot \frac{1}{\base^{j}-1}
\end{eqnarray*}
where $\omega_1\ldots \omega_i$ and $\omega_{i + 1}\ldots \omega_{i + j}$ are integers that have been written in base $\base$. Adding the two resulting fractions together, we end up with a (complicated) expression whose denominator is $\base^i(\base^{j}-1)$. Further cancellations may or may not be possible, but we can \emph{always} write the rational number as a fraction of two integers, the denominator of which is $\base^i(\base^{j}-1)$.

This fact leads to a natural height function on $F\cap \Q$ related to the base $\base$ structure of the fractal $F$:
\begin{equation}\label{eqsymbheight}
H_\sym(r) \df \base^i \cdot (\base^{j}-1),
\end{equation}
where the indices $i$ and $j$ are the smallest integers such that $r$ can be written in the form \eqref{equationrationalform}. The function $H_\sym$ is called the \emph{symbolic} height function. It was studied in a more general context in \cite{FishmanSimmons2}. Notice the symbolic height of a rational number may not be the same as its standard height (i.e. its denominator in reduced form). For example, the rational number
$0.\overline{20\vphantom{I}}_{3}$ in the Cantor ternary set is equal to $\frac{3}{4}$, so its standard height is $4$. Nonetheless, the symbolic height of $0.\overline{20\vphantom{I}}_3$ is $3^0\cdot (3^2 - 1) = 8$. It should be thought of as the denominator resulting from the following calculation:
\begin{eqnarray*}
0.\overline{20}_3 &=& \frac{20_3}{3^0} \sum^\infty_{m=1} \left(\frac{1}{3^2}\right)^m \quad = \quad \frac{6}{1} \cdot \frac{1/3^2}{1 - 1/3^2}\\ &=& \frac{6}{1} \cdot \frac{1/9}{8/9} \quad = \quad \frac{6}{8}\cdot
\end{eqnarray*}
Although more cancellation is possible at the end of this calculation, this will not always be the case,\Footnote{For example, the fraction at the end of the calculation
\[
0.2\overline{70}_9 = \frac{2_9}{9} + \frac{70_9}{9} \cdot \frac{1}{9^2 - 1} = \frac{2\cdot 80 + 7\cdot 9}{9\cdot 80} = \frac{223}{720}
\]
is already in reduced form.} so in a principled way we have stopped reducing the fraction here. The calculation illustrates the fact that the symbolic height of a rational number $r$ can be thought of as a ``symbolic denominator'', i.e. the denominator of a certain representation of $r$ as the quotient of two integers. The numerator of this representation can be thought of as a ``symbolic numerator'' (in the above example the symbolic numerator would be 2), but as usual, for purposes of Diophantine approximation it is simpler to just work with the denominator. Note that the standard height is by definition smaller than the symbolic one, since we have $p_\std/q_\std = p_\sym/q_\sym$, but the left-hand side is in reduced form.

We remark that heuristically, if we are given two rational numbers $r_1$ and $r_2$, and we are told that $r_1$ lies in the limit set of a base $\base$ IFS, but we are not told anything about $r_2$, then we should expect the (multiplicative) discrepancy between the standard height and the symbolic height to be smaller for $r_1$ than for $r_2$. This is because if we choose the numerator and denominator of a rational randomly, then the numbers $i$ and $j$ satisfying \eqref{equationrationalform} may be comparable to the standard height of the rational (meaning that the symbolic height is an exponential function of the standard height), but the number would be exceedingly unlikely to lie in any base $\base$ limit set, since its digits would essentially be random. By contrast, if we choose the digits of a rational randomly out of a fixed alphabet $\alphabet$ (with a fixed period and preperiod), then the amount of cancellation we expect to see in the symbolic representation of the rational will be much smaller, so the standard height and symbolic height will be relatively close. More heuristics regarding the relation between the symbolic height function and the standard one were discussed in \cite{FishmanSimmons2}.

One reason the symbolic height function is interesting is that it naturally shows up in the proofs of results regarding the standard height function. For example, the proof of Proposition \ref{BFRresult} can easily be modified to bound $|x - p/q|$ in terms of the symbolic height of $p/q$ rather than the standard height:

\begin{proposition}[{\cite[Proof of Corollary 2.2]{BFR}}]
\label{BFRresult2}
Let $F$ be the limit set of a base $\base$ IFS, and let $\delta$ be the Hausdorff dimension of $F$. Then for all $x\in F$, there exist infinitely many rational numbers $r = p_\sym/q_\sym\in F$ such that
\[
\bigl|x-p/q\bigr| < \frac{1}{q_\sym(\log_\base q_\sym)^{1/\delta}}\cdot
\]
In other words, the function $\psi_*(q) = (q\cdot (\log_\base q)^{1/\delta})^{-1}$ is an intrinsic Dirichlet function on $F$ for the symbolic height function.
\end{proposition}

In fact, the proof of \cite[Corollary 2.2]{BFR} essentially proceeds by first proving Proposition \ref{BFRresult2} and then using the inequality $H_\std \leq H_\sym$ to deduce Proposition \ref{BFRresult}. It appears extremely difficult to prove any improvement (either for all points or only for some) of Proposition \ref{BFRresult} for the standard height without just proving the same bound for the symbolic height. So in some way, the symbolic height is measuring the ``strength of our techniques''.

Although the symbolic height function is motivated in terms of the standard height function, it can also be analyzed on its own terms. For example, we can ask whether the intrinsic Dirichlet function $\psi_*$ appearing in Proposition \ref{BFRresult2} is optimal for the symbolic height function. This is the same (cf. \cite[\62.1]{FishmanSimmons5}) as asking whether there exist any points in $F$ that are badly symbolically approximable with respect to $\psi_*$:

\begin{definition}[Special case of {\cite[Definition 4.7]{FishmanSimmons2}}]
Let $F$ be a base $\base$ limit set, and let $\delta$ denote the Hausdorff dimension of $F$. A number $x\in F$ is called \emph{badly symbolically approximable (with respect to $\psi_*$)} if there exists $\kappa > 0$ such that for every $r = p_\sym/q_\sym \in F\cap \Q$, we have
\begin{equation}
\label{BAdef}
|x-r| \geq \frac{\kappa}{{q_\sym(\log_\base q_\sym)^{1/\delta}}}.
\end{equation}
\end{definition}

\begin{theorem}[Corollary of {\cite[Lemma 4.9]{FishmanSimmons2}}; or see below]
\label{BAintrinsic}
Let $F$ be the limit set of a base $\base$ IFS satisfying the strong separation condition. Then any $x \in F$ whose base $\base$ expansion is uniformly de Bruijn is badly symbolically approximable.
\end{theorem}

Combining with Theorem \ref{MainTheoremDeBruijn} gives:

\begin{corollary}
\label{corollaryoptimality}
With $F$ as above, the set of badly symbolically approximable points has dimension at least $\alpha_k \delta > 0$, where
\[
\alpha_k = \begin{cases}
1/49 & k=2\\
(8\cdot (9\log_4(3) - 1))^{-1} & k= 3\\
\frac{\log(k-2)!}{k\log(k)} & k \geq 4 \end{cases}.
\]
In particular, the intrinsic Dirichlet function $\phi_*$ appearing in Proposition \ref{BFRresult2} is optimal.
\end{corollary}

We remark that the optimality assertion follows directly from combining Theorem \ref{BAintrinsic} with \cite[Corollary 7]{BecherHeiber}; Theorem \ref{MainTheoremDeBruijn} is not needed.

In contrast to Proposition \ref{BFRresult2}, Theorem \ref{BAintrinsic} and Corollary \ref{corollaryoptimality} are weaker than their (unproven) analogues for the standard height function. This is because while Proposition \ref{BFRresult2} is about finding good approximations to points, in Theorem \ref{BAintrinsic} and Corollary \ref{corollaryoptimality} we show that for certain points, good approximations cannot exist. But the inequality $H_\std \leq H_\sym$ means that the quality of an approximation is better according to the standard height than according to the symbolic height, which yields the appropriate implications.

We remark that Theorem \ref{BAintrinsic} is only a one-way implication: there may be (and almost certainly are) badly symbolically approximable numbers whose base $\base$ expansions are not uniformly de Bruijn. A combinatorial characterization of the base $\base$ expansions of badly symbolically approximable numbers was given in \cite[Lemma 4.9]{FishmanSimmons2}. As a consequence of the one-sidedness of the implication, Theorem \ref{BAintrinsic} yields a lower bound on the dimension of the set of badly symbolically approximable points but not an upper bound. In fact, we believe that there is no nontrivial upper bound: we conjecture that the Hausdorff dimension of the set of badly symbolically approximable points of any base $\base$ limit set $F$ is equal to the Hausdorff dimension of $F$. This conjecture is motivated by other situations in Diophantine approximation where the dimension of the set of badly approximable points has always turned out to be full. However, Theorem \ref{MainTheoremDeBruijn} shows that this conjecture cannot be proven using uniformly de Bruijn sequences.


Although Theorem \ref{BAintrinsic} is a consequence of the much more general result \cite[Lemma 4.9]{FishmanSimmons2}, we prove it here for completeness and ease of exposition.

\begin{proof}[Proof of Theorem \ref{BAintrinsic}]
Let $x \in F$ be a number whose base $\base$ expansion, which we denote by $\omega$, is uniformly de Bruijn. Let $\ell$ denote the size of the largest gap in the set $B_\omega$ defined by \eqref{deBruijn}. Fix $r \in F \cap \Q$, and let the representation $r = 0.\tau_1\dots \tau_i \overline{\tau_{i+1} \dots \tau_j\vphantom{t}}$ be chosen so as to minimize $i$ and $j$. Then the symbolic height of $r$, as defined in \eqref{eqsymbheight}, is $q_\sym = \base^i(\base^{j-i} - 1) \leq \base^j$. Since the IFS defining $F$ is assumed to satisfy the strong separation condition, the distance between $x$ and $r$ is comparable to $\base^{-m}$, where $m$ is the largest index for which $\omega_i = \tau_i$ for all $i \leq m$. In fact, a careful analysis shows that $|x - r| \geq \base^{-(m + 2)}$, though the precise constant factor is not relevant. We claim that if $j \geq \ell$, then
\begin{equation}
\label{ETSBAintrinsic}
\base^{-m} \geq \frac{\base^{-\ell}}{\base^j j^{1/\delta}},
\end{equation}
which demonstrates that \eqref{BAdef} holds with $\kappa = \base^{-(\ell + 2)}$. We now separate into two cases:

{\bf Case 1:} $m \leq j + \ell$. In this case, we have
\[
\base^{-m} \geq \base^{-j} \base^{-\ell} \geq \frac{\base^{-\ell}}{\base^j j^{1/\delta}},
\]
as required.

{\bf Case 2:} $m > j + \ell$. In this case, by the $m$th letter, the sequence $\tau$ will have already begun to repeat. The longest repeated string in the sequence $\tau_1\ldots\tau_m$ is $\tau_{i+1}\dots \tau_{m-(j-i)} = \tau_{j+1} \dots \tau_{m}$. Note that although the two sides of this equation represent distinct instances of the same string as a substring of $\tau_1\ldots\tau_m$, the two instances may overlap with each other; this happens if and only if $m > 2j - i$. For the purposes of our calculations, it does not matter whether these two instances overlap or not.

By the definition of $m$, we have $\omega_1\cdots \omega_m = \tau_1\cdots \tau_m$, so $\omega$ also has a repeated string $\omega_{i+1}\dots \omega_{m-(j-i)} = \omega_{j+1} \dots \omega_{m}$ of length $(m - j)$ occurring in the first $m$ letters. On the other hand, by the definition of $\ell$, there exists $m - j - \ell < n \leq m - j$ such that $n \in B_\omega$, which implies that $\omega$ has no repeated string of length $n$ occurring in the first $k^n + n - 1$ letters of $\omega$. Since $n \leq m - j$, it follows that $m > k^n + n - 1$, and thus
\[
k^n \leq m - n < j + \ell \leq 2j.
\]
Since $k \geq 2$ and $n \geq m - j - \ell + 1$, this implies
\[
k^{m - j - \ell} \leq j.
\]
Raising both sides to the power of $1/\delta$ gives
\[
\base^{m - j - \ell} \leq j^{1/\delta},
\]
and rearranging gives \eqref{ETSBAintrinsic}.
\end{proof}

\bibliographystyle{amsplain}

\bibliography{bibliography}

\providecommand{\bysame}{\leavevmode\hbox to3em{\hrulefill}\thinspace}
\providecommand{\MR}{\relax\ifhmode\unskip\space\fi MR }
\providecommand{\MRhref}[2]{%
  \href{http://www.ams.org/mathscinet-getitem?mr=#1}{#2}
}
\providecommand{\href}[2]{#2}
\begin{thebibliography}{10}

\bibitem{Aigner}
Martin Aigner, \emph{A course in enumeration}, Graduate Texts in Mathematics,
  vol. 238, Springer, Berlin, 2007. \MR{2339282}

\bibitem{BecherHeiber}
Ver{\'o}nica Becher and Pablo Heiber, \emph{On extending de {B}ruijn
  sequences}, Inform. Process. Lett. \textbf{111} (2011), no.~18, 930--932.
  \MR{2849850 (2012e:68258)}

\bibitem{BernikDodson}
Vasilii Bernik and Maurice Dodson, \emph{Metric {D}iophantine approximation on
  manifolds}, Cambridge Tracts in Mathematics, vol. 137, Cambridge University
  Press, Cambridge, 1999.

\bibitem{BFKRW}
Ryan Broderick, Lior Fishman, Dmitry Kleinbock, Asaf Reich, and Barak Weiss,
  \emph{The set of badly approximable vectors is strongly {$C^1$}
  incompressible}, Math. Proc. Cambridge Philos. Soc. \textbf{153} (2012),
  no.~02, 319--339.

\bibitem{BFR}
Ryan Broderick, Lior Fishman, and Asaf Reich, \emph{Intrinsic approximation on
  {C}antor-like sets, a problem of {M}ahler}, Mosc. J. Comb. Number Theory
  \textbf{1} (2011), no.~4, 291--300.

\bibitem{Bugeaud}
Yann Bugeaud, \emph{Approximation by algebraic numbers}, Cambridge Tracts in
  Mathematics, vol. 160, Cambridge University Press, Cambridge, 2004.

\bibitem{BugeaudDurand}
Yann Bugeaud and Arnaud Durand, \emph{Metric {D}iophantine approximation on the
  middle-third {C}antor set}, J. Eur. Math. Soc. (JEMS) \textbf{18} (2016),
  no.~6, 1233--1272. \MR{3500835}

\bibitem{EFS}
Manfred Einsiedler, Lior Fishman, and Uri Shapira, \emph{{D}iophantine
  approximations on fractals}, Geom. Funct. Anal. \textbf{21} (2011), no. 1,
  14--35.

\bibitem{Falconer_book}
Kenneth Falconer, \emph{Fractal geometry: {M}athematical foundations and
  applications}, John Wiley \& Sons, Ltd., Chichester, 1990.

\bibitem{Fishman}
Lior Fishman, \emph{{S}chmidt's game on fractals}, Israel J. Math. \textbf{171}
  (2009), no. 1, 77--92.

\bibitem{FishmanSimmons1}
Lior Fishman and David Simmons, \emph{Intrinsic approximation for fractals
  defined by rational iterated function systems - {M}ahler's research
  suggestion}, Proc. Lond. Math. Soc. (3) \textbf{109} (2014), no. 1, 189--212.

\bibitem{FishmanSimmons2}
\bysame, \emph{Extrinsic {D}iophantine approximation on manifolds and
  fractals}, J. Math. Pures Appl. (9) \textbf{104} (2015), no.~1, 83--101.

\bibitem{FishmanSimmons5}
Lior Fishman, David Simmons, and Mariusz Urba{\'n}ski, \emph{{D}iophantine
  approximation in {B}anach spaces}, J. Th{\'e}or. Nombres Bordeaux \textbf{26}
  (2014), no.~2, 363--384.

\bibitem{KLW}
Dmitry Kleinbock, Elon Lindenstrauss, and Barak Weiss, \emph{On fractal
  measures and {D}iophantine approximation}, Selecta Math. \textbf{10} (2004),
  479--523.

\bibitem{KleinbockWeiss1}
Dmitry Kleinbock and Barak Weiss, \emph{Badly approximable vectors on
  fractals}, Israel J. Math. \textbf{149} (2005), 137--170.

\bibitem{Legendre}
Adrien-Marie Legendre, \emph{Essai sur la th\'eorie des nombres ({E}ssay on
  number theory). {R}eprint of the second (1808) edition}, Cambridge Library
  Collection, Cambridge University Press, Cambridge, 2009 (French).

\bibitem{Mahler}
Kurt Mahler, \emph{Some suggestions for further research}, Bull. Aust. Math.
  Soc. \textbf{29} (1984), 101--108.

\bibitem{Weiss}
Barak Weiss, \emph{Almost no points on a {C}antor set are very well
  approximable}, R. Soc. Lond. Proc. Ser. A Math. Phys. Eng. Sci. \textbf{457}
  (2001), no. 2008, 949--952.

\end{thebibliography}

\end{document}